\documentclass[12pt]{article}
\usepackage[latin1]{inputenc}
\usepackage{amsmath}
\usepackage{amsfonts}
\usepackage{amssymb}
\usepackage{graphicx}
\usepackage{color}

\newtheorem{theorem}{Theorem}[section]

\newtheorem{lemma}[theorem]{Lemma}

\newtheorem{remark}[theorem]{Remark}

\def\P{\mathbb{P}}

\def\e{\mathbf{E}}

\newcommand{\ter}{\hspace{\stretch{3}}$\square$\\[1.8ex]}

\def\l{\langle}

\def\r{\rangle}

\title{\bf An application of the backbone decomposition to supercritical super-Brownian motion with a barrier}
\author{{\large A. Kyprianou\footnote{{\sc Department of Mathematical Sciences, University of Bath, Claverton Down, Bath, BA2 7AY, UK.} 
}
\ \  A. Murillo-Salas\footnote{{\sc  Departamento de Matem\'aticas, Universidad de Guanajuato,
Jalisco s/n, Mineral de Valenciana,
Guanajuato, Gto. C.P. 36240, M\'exico.} 
}} \ \ {\large and} \  J. L.
P\'erez$^{*,}$\footnote{{\sc Department of Statistics, ITAM, Rio Hondo 1,
Tizapan 1 San Angel, 01000 M\'exico, D.F.}
}
}
\begin{document}

\maketitle
\begin{abstract}\noindent We analyse the behaviour of supercritical super-Brownian motion with a barrier through the pathwise backbone embedding of Berestycki et al. \cite{BKMS}. In particular, by considering existing results for branching Brownian motion due to Harris et al. \cite{HHK} and Maillard \cite{Mal}, we obtain, with relative ease, conclusions regarding the growth in the right most point in the support, analytical properties of the associated one-sided FKPP {\color{black}wave} equation as well as the distribution of mass on the exit measure associated with the barrier.
\bigskip

\noindent {\sc Key words and phrases}: Super-Brownian motion,
backbone decomposition, killed super-Brownian
motion.

\bigskip

\noindent MSC 2010 subject classifications: 60J68, 35C07.
\end{abstract}

\vspace{0.5cm}

\section{Introduction}

Suppose that $X=\{X_t: t\geq 0\}$ is a (one-dimensional) super-diffusion   with motion corresponding to that of a Brownian motion with drift $-\rho\in\mathbb{R}$, {\color{black} stopped} at zero, and branching mechanism $\psi$ taking the form
\begin{equation}
\psi(\lambda) = -\alpha \lambda + \beta\lambda^2 + \int_{(0,\infty
)} (e^{-\lambda x} - 1 + \lambda x )\Pi({\rm d}x),
\label{branch-mech}
\end{equation}
for $\lambda \geq 0$ where $\alpha = - \psi'(0^+)\in(0,\infty)$,
$\beta\geq 0$ and $\Pi$ is a measure concentrated on $(0,\infty)$
which satisfies
$\int_{(0,\infty)}(x\wedge x^2)\Pi({\rm d}x)<\infty$. {\color{black} We also insist that $\beta>0$ if $\Pi\equiv 0$.} 
The existence of this class of superprocesses is guaranteed by \cite{D, Dyn1993, Dyn2002}.

Let
$\mathcal{M}_F(I)$ be  the space of finite measures on $I\subseteq \mathbb{R}$
and note that $X$ is a $\mathcal{M}_F[0,\infty)$-valued Markov
process under $\mathbb{P}_\mu$ for each $\mu\in\mathcal{M}_F[0,\infty)$, where $\mathbb{P}_\mu$ is law of $X$ with initial configuration $\mu$. One may think of $\mathbb{P}_\mu$ as a law on cadlag path-space $\mathcal{X} := D([0,\infty)\times \mathcal{M}_F[0,\infty))$.  Henceforth we shall use standard inner product notation, for
$f\in C_b^+[0,\infty)$ and $\mu\in\mathcal{M}_F[0,\infty)$,
\[
\langle f , \mu\rangle =  \int_{\mathbb{R} }f(x)\mu({\rm d}x).
\]
Accordingly we shall write $||\mu|| = \langle 1,\mu \rangle$.



Recall that the total mass of the process $X$ is a continuous-state branching process with branching mechanism $\psi$. Since there is no interaction between spatial motion and branching
we can characterise our  $\psi$-superdiffusion as supercritical on account of the assumption  $\alpha = -\psi'(0^+)>0$.  Such processes may exhibit explosive behaviour, however, under the conditions assumed above, $X$ remains finite at all positive times.
We insist moreover that $\psi(\infty)=\infty$ which means that with positive probability the event $\lim_{t\uparrow\infty}||X_t||=0$ will occur. Equivalently this means that the total mass process does not have monotone increasing paths; see for example the summary in Chapter 10 of Kyprianou \cite{K}. 
The probability of the
event
\[
\mathcal{E}:= \{\lim_{t\uparrow \infty} ||X_t|| =0\}
\]
 is described in terms of the largest root, say $\lambda^*$, of
the equation $\psi(\lambda)=0$. It is known (cf. Chapter 8 of
\cite{K}) that $\psi$ is strictly convex with $\psi(0)=0$ and
hence since $\psi(\infty)=\infty$ and $\psi'(0^+)<0$ it follows that
there are exactly two roots in $[  0,\infty)$, one of which is always $0$. For $\mu\in
\mathcal{M}_F[0,\infty)$ we have
\begin{equation}
 \mathbb{P}_\mu (\lim_{t\uparrow \infty} ||X_t|| =0) = e^{-\lambda^* ||\mu||}.
 \label{extinguishpr}
\end{equation}
It is a straightforward exercise (cf. Lemma 2 of \cite{BKMS} or Theorem 2.6 of \cite{Sheu}) to show that the law of $X$ under $\mathbb{P}_\mu$  conditioned on $\mathcal{E}$ is that of another super-diffusion with the same motion component as $X$, but with a new branching mechanism which is given by $\psi^*(\lambda)  = \psi(\lambda+ \lambda^*)$ for $\lambda\geq 0$. Said another way, the aforesaid super-diffusion has semigroup characterised by the non-linear equation (\ref{PDE}) with the quantity $\psi$ replaced by $\psi^*$. We denote its law by $\mathbb{P}^{*}_\mu$.

In this article we shall also assume  that
  \begin{equation}
 \int^\infty\frac{1}{\sqrt{\int_{\lambda^*}^\xi \psi(u) {\rm d}u}}{\rm d}\xi<\infty. 
 \label{A3}
 \end{equation}
The condition (\ref{A3}) implies in particular that $\int^\infty 1/\psi(\xi){\rm d}\xi<\infty$ (cf. \cite{Sheu}) which in turn guarantees that the event
$
\mathcal{E}
$
 agrees with the event of {\it extinction}, namely $\{\zeta^{X}<\infty\}$ where
\[
 \zeta^{X} = \inf\{t> 0 : ||X_t || = 0\}.
\]
 Note that (\ref{A3}) cannot be satisfied for branching mechanisms which belong to bounded variation spectrally positive L\'evy processes.

\bigskip

In this paper our objective is to show the robustness of a recent pathwise backbone decomposition, described in detail in the next section, as a mechanism for transferring results from branching diffusions directly into the setting of superprocesses. We shall do this by demonstrating how two related fundamental results for branching Brownian motion with a killing barrier induce the same results for a  $\psi$-super-Brownian motion with killing at the origin. The latter, which we shall denote by $X^+=  \{X^+_t: t\geq 0\}$, can be defined on the same probability space as $X$ by simply taking
\begin{equation}
X^+_t(\cdot) = X_t(\cdot \cap (0,\infty)).
\label{notation}
\end{equation}
For $f\in C^+_b(0,\infty)$, $\mu\in\mathcal{M}_F(0,\infty)$, $x>0$ and $t\geq 0$,
\begin{equation}
 -\log\mathbb{E}_\mu(e^{- \langle f, X_t\rangle}) =  \int_{(0,\infty)}u_f(x, t)\mu({\rm d}x), \, \, t\geq 0,
 \label{prePDE}
\end{equation}
 describes the semi-group of {\color{black} $X$}, where $u_f$ is the unique positive solution to
 {\color{black}
       \begin{equation}\label{PDE}
 u_f(x,t)={\rm E}_x^{-\rho}[f(B_{t\wedge\tau_0})] - {\rm E}_x^{-\rho}\left[\int_0^{t\wedge\tau_0} \psi(u_f(B_s, t-s)){\rm d}s\right] \qquad x,t\geq 0.
       \end{equation}
Here, ${\rm E}^{-\rho}_x$ is expectation with respect to ${\rm P}^{-\rho}_x$, under which $\{B_t: t\geq 0\}$ is a Brownian motion with  drift $-\rho$, issued from $x\geq 0$ and $\tau_0 = \inf\{t>0 : B_t <0\}$.} 
The reader is referred to Theorem 1.1 of  Dynkin \cite{Dyn1991}, Proposition 2.3 of Fitzsimmons \cite{Fitz} and Proposition 2.2 of Watanabe \cite{watanabe1968} for further details; see also Dynkin  \cite{Dyn1993, Dyn2002} for a general overview.

\bigskip

Our first result, based on the branching particle analogue in \cite{HHK},  shows that the  classical growth of the right most point in the support and its intimate {\color{black} relation with non-negative stationary solutions to (\ref{PDE}) can also be seen in the superprocess context. Specifically, we mean solutions of the form $u(x,t) = \Phi(x)$, which necessarily solve
\[
\Phi(x) = {\rm E}^{-\rho}_x[\Phi(B_{t\wedge\tau_0})] - {\rm E}_{x}^{-\rho}\left[\int_0^{t\wedge \tau_0} \psi(\Phi(B_s)){\rm d}s\right], \qquad x\geq 0.
\]
If we additionally suppose, for technical reasons which are soon to become apparent,  that
$\Phi$ monotonically  connects zero the origin to $\lambda^*$ at $+\infty$, then it is a straightforward exercise using classical Feynman-Kac representation of solutions to ODEs
to show that $\Phi$ solves the differential equation
\begin{equation}
\frac{1}{2}\Phi''(x) - \rho\Phi'(x) - \psi(\Phi(x)) =0 \text{ on }x>0\text{ with }
\Phi (0+) = 0 \text{ and }\Phi(+\infty) = \lambda^*.
\label{FKPP}
\end{equation}
In that case we call $\Phi$ a wave solution to (\ref{FKPP}).}

\begin{theorem}[Strong law for the support]\label{I}
Define
\begin{equation}
R^{X}_t:=\inf\{y>0:X_t(y,\infty)=0\}=\inf\{y>0:X^+_t(y,\infty)=0\}
\label{rightmost-X}
\end{equation}
and
denote the extinction time of $X^+$ by
\[
\zeta^X_+= \inf\{t>0 : ||X^+_t|| =0\}.
\]
\begin{itemize}
\item[(i)] Assume that $-\infty<\rho<\sqrt{2\alpha}$. Then, for all $x>0$,
\begin{equation}\label{RMP1}
 \lim_{t\rightarrow\infty}\frac{R^{X}_t}{t}= \sqrt{2\alpha}-\rho\text{ on } \{\zeta^{X}_+=\infty\},
\end{equation}
 $\P_{\delta_x}$-almost surely and
 \begin{equation}
\Phi(x):=-\log \P_{\delta_x}(\zeta^{X}_+<\infty),\,\mbox{for all }\, x>0.
\end{equation}
is the unique wave solution to (\ref{FKPP}).

\item[(ii)]  For all $\rho\geq \sqrt{2\alpha}$ there exists no monotone wave solution to (\ref{FKPP}) and
\[
\P_{\delta_x}(\zeta^{X}_+<\infty)=1\mbox{ for all } x>0.
\]
 \end{itemize}
\end{theorem}


\begin{remark}\rm
Whilst Theorem \ref{I}  offers results on the existence and uniqueness of solutions to (\ref{FKPP}) we do not claim that these are necessarily new.
Indeed one may extract the same or similar results using the methods in, for example Kamataka \cite{Kam}, Uchyama \cite{U78} and Pinsky \cite{Pin}.  See also the discussion in Remark \ref{trick} below.
\end{remark}

Our second result looks at the distribution of mass that is absorbed at the origin, when $\rho$ takes the critical value $\sqrt{2\alpha}$,  in the spirit of recent results of Addario-Berry and Broutin \cite{AB}, A\"id\'ekon et al. \cite{AHZ} and Maillard \cite{Mal}.  In order to describe this result we need to introduce the concept of Dynkin's exit measures.

{\color{black}For each $x\in\mathbb{R}$,  suppose we defined the superprocess $Y  = \{Y_t : t\geq 0\}$ under $\mathbb{Q}_{\delta_x}$ to have the same branching mechanism as  $(X, \mathbb{P}_{\delta_x})$ however, the underlying motion associated with $Y$ is that of a Brownian motion with drift $-\rho$ (i.e. no stopping at $0$). The existence of $(Y, \mathbb{Q}_{\delta_x})$ is justified through the same means as for $(X, \mathbb{P}_{\delta_x})$. In principle it is possible to construct these two processes on the same probability space, however, this is unnecessary for our purposes. For each $z,t\geq 0$, define the space-time domain $D^t_{-z} = \{(x,u) \in \mathbb{R}\times[0,\infty): u<t, x>-z\}$.
According to Dynkin's theory of exit measures {\color{black}outlined in}
Section 7 of \cite{Dyn2001} and Section 1 of \cite{Dynkin-Kuznetsov}, it is possible to describe the mass in the
superprocess $Y$ as it first exits the  domain $D^t_{-z}$. 
In particular, according to the characterisation  for branching Markov exit measures given in Section 1.1 of \cite{Dynkin-Kuznetsov}, 
the random measure $Y_{D^t_{-z}}$ is
 supported on $\partial D^t_{-z} = (\{-z\}\times[0,t))\cup([-z,\infty)\times\{t\})$ and is characterised by the Laplace functional
 \[
 \mathbb{Q}_{\delta_x}(e^{-\l f,Y_{D^t_{-z}}\r}) = e^{ - u^z_f(x,t) },
 \]
 where $x\geq -z$, $f\in C_b([-z,\infty)\times[0,\infty))$ and $u^z_f(x,t)$ uniquely solves, amongst non-negative solutions, (cf. Theorem 6.1 of \cite{Dyn2001}) the equation
\begin{equation}
   u^z_f(x,t)  ={\rm E}^{-\rho}_{x}[f(B_{t\wedge \tau_{-z}}, t\wedge \tau_{-z})]
   - {\rm E}^{-\rho}_x\left[\int_0^{t\wedge \tau_{-z}}
   \psi(u^z_f(B_u, t-u))
    {\rm d}u\right]\qquad x\geq -z, t\geq 0,
       \label{PDE3a}
\end{equation}
where $\tau_{-z} = \inf\{t>0 : B_t<-z\}$. Intuitively speaking, one should think of $Y_{D^t_{-z}}$ as the analogue of the atomic measure supported on $\partial D^t_{-z} $ which describes the collection of particles and their space-time position in a branching Brownian motion with drift $-\rho$ who are  first in their  genealogical line of descent to exit the space-time domain $(-z,\infty)\times[0,t)$.
}

In the case that $\rho\geq \sqrt{2\alpha}$, it was shown in Theorem 3.1 of \cite{KLMSR}  that the the limiting random measure $Y_{D_{-z}} = \lim_{t\uparrow\infty}Y_{D^t_{-z}}$ (which exists almost surely by monotonicity) is almost surely finite and has total mass which satisfies
\[
\mathbb{Q}_{\delta_x} (e^{- \theta||Y_{D_{-z}}|| }) = e^{- v_\theta(x+z)},
\]
for $\theta\geq 0$, $x\geq -z$, where
\[
\frac{1}{2}v_\theta''(x) - \rho v_\theta'(x) - \psi(v_\theta(x)) = 0,
\]
with $v_\theta(0) = \theta$.
In particular, $\{v_\theta(x):x\geq 0\}$ is the semigroup
of a continuous-state branching process with branching mechanism which satisfies
\[
\psi_{D}(\lambda) = \Psi'(\Psi^{-1}(\lambda)),
\]
for $\lambda\in[0,\lambda^*]$, where
 $\Psi$ is the unique monotone solution to the wave equation
\begin{equation}
\frac{1}{2}\Psi''(x) + \rho\Psi'(x) - \psi(\Psi(x)) = 0 \text{ on }\mathbb{R}\text{ with }\Psi(-\infty) = \lambda^* \text{ and }\Psi(+\infty) = 0.
\label{PSI}
\end{equation}
 Indeed, it was shown in Theorem 3.1 of \cite{KLMSR} that $||Y_D||: = \{||Y_{D_{-z}}||: z\geq 0\}$ is a continuous-state branching process with growth rate $\rho + \sqrt{  \rho^2- 2\alpha}$.

We are now ready to state our second main result,  based on the branching Brownian motion  analogue in \cite{Mal}, which in particular focuses on the case that the underlying motion has a critical speed $\sqrt{2\alpha}$.

\begin{theorem}[Absorbed mass at criticality]\label{III} Set $\rho = \sqrt{2\alpha}$.
Assume that for some $\varepsilon>0$,
\begin{equation}
\int_{[1,\infty)}x(\log x)^{2+\varepsilon}\Pi({\rm d}x)<\infty.
\label{Pi-moment}
\end{equation}
Then for each $z,x>0$ we have
\begin{equation}
\mathbb{Q}_{\delta_x}(||Y_{D_{-z}}||>t)\sim\sqrt{2\alpha}\frac{(x+z)e^{(x+z)\sqrt{2\alpha}}}{t(\log t)^2}.\notag
\end{equation}
as $t\uparrow\infty$.
\end{theorem}

Note that in terms of our earlier notation, we see that,  $X^+_t$ under $\mathbb{P}_{\delta_x}$ has the same law as $Y_{D^t_{0}}|_{(0,\infty)\times \{t\}}$ under $\mathbb{Q}_{\delta_x}$. Whilst Theorem \ref{I} therefore concerns the spatial evolution of the support of the measure $Y_{D^t_0}$ away from the origin for  speeds $\rho> \sqrt{2\alpha}$, by contrast Theorem \ref{III} above addresses the distribution of mass accumulated at the origin by the same measure, at the critical speed $\sqrt{2\alpha}$.

\bigskip

The remainder of this paper is structured as follows. In the next section we give a brief overview of the backbone decomposition for $X$, noting that similar decompositions also hold for a number of other processes used in this article. In Section \ref{pfI} we prove Theorem \ref{I} and in Section \ref{pfIII} we prove Theorem \ref{III}.

\section{The backbone decomposition and Poissonisation}

As alluded to above, our results are largely driven by the backbone decomposition, recently described in the pathwise sense by \cite{BKMS} for {\it conservative} processes. Note that  backbone decompositions have been known in the earlier and more analytical setting of semigroup decompositions through the work of \cite{EO} and \cite{EP} as well as in the pathwise setting in the work of \cite{SV1, SV2}.


To describe the backbone decomposition in detail, 
consider the process $\{\Lambda^X_t : t \geq 0\}$ which has the following pathwise construction.
First sample from a branching particle diffusion with branching generator
\begin{equation}
F(r) = q\left(\sum_{n\geq 0} p_n r^n - r\right) =  \frac{1}{\lambda^*}\psi(\lambda^*(1-r)), \, r\in[0,1],
\label{F}
\end{equation}
and particle motion which is that of a Brownian motion {\color{black} with drift $-\rho$}, {\color{black} stopped} at the origin.
Note that in the above generator, we have that $q$ is the rate at which individuals reproduce and $\{p_n: n\geq 0\}$ is the offspring distribution. With the particular branching generator given by \eqref{F}, $q = \psi'(\lambda^*)$, $p_0 = p_1 =0$, and for $n\geq 2$,  $p_n : = p_n[0,\infty)$ where for $y\geq 0$, we defined the measure $p_n(\cdot)$ on $\{2,3,4,\ldots\}\times[0,\infty)$ by
\[
p_n({\rm d}y) =  \frac{1}{\lambda^* \psi'(\lambda^*)}\left\{\beta (\lambda^*)^2\delta_0({\rm d}y)\mathbf{1}_{\{n=2\}} + (\lambda^*)^n \frac{y^n}{n!} e^{-\lambda^*y} \Pi({\rm d}y)\right\}.
\]
If we denote the aforesaid branching particle diffusion by  $Z^X = \{Z^X_t: t\geq 0\}$ then we shall also insist that the configuration of particles in space at time zero, $Z_0$, is given by an independent    Poisson random measure with intensity $\lambda^*\mu$.
Next, 
{\it dress} the branches of the spatial tree that describes
the trajectory of $Z^X$ in such a way that a particle at the
space-time position $(\xi, t)\in[0,\infty)^2$ has an independent
$\mathcal{X}$-valued trajectory grafted on to it with rate
\[
2\beta {\rm d}\mathbb{N}_\xi^* + \int_0^\infty y e^{-\lambda^* y}\Pi({\rm d}y){\rm d}\mathbb{P}^{*}_{\xi\delta_y}.
\]
Here the measure $\mathbb{N}_\xi^*$ is the excursion measure (cf. \cite{Dynkin-Kuznetsov, LG, El}) on the space $\mathcal{X}$  which satisfies
\[
\mathbb{N}_x^*(1- e^{-\l f,X_t\r}) = u^*_f(x,t),
\]
for $x,t\geq 0$ and $f\in C^+_b[0,\infty)$, where $ u^*_f(x,t)$ is the unique solution to (\ref{PDE}) with the branching mechanism $\psi$ replaced by $\psi^*$.
 Moreover, on the event that an individual in $Z^X$ dies and branches into $n\geq 2$ offspring at spatial position $\xi\in[0,\infty)$, with probability $p_n({\rm d}y)\mathbb{P}^{*}_{y\delta_\xi}$, an additional independent $\mathcal{X}$-valued trajectory is grafted on to the space-time branching point. The quantity $\Lambda^X_t$ is now understood to be the total dressed mass present at time $t$ together with the mass present at time $t$ of an independent copy of $(X,\mathbb{P}^{*}_{\mu})$ issued at time zero.
We denote the law of $(\Lambda^X, Z^X)$ by $\mathbf{P}_{\mu}$.

The backbone decomposition is now summarised by the following theorem lifted from Berestycki et al. \cite{BKMS}.

\begin{theorem}\label{main-1}
For any $\mu\in\mathcal{M}_F(\mathbb{R}^d)$, the process $(\Lambda^X,  \mathbf{P}_\mu)$ is Markovian and has the same law as  $(X,  \mathbb{P}_\mu )$. Moreover, for each $t\geq 0$, the law of $Z^X_t$ given $\Lambda^X_t$ is that of a Poisson random measure with intensity measure $\lambda^*\Lambda^X_t$.
\end{theorem}

Not much changes in the above account when we replace the role of $X$ by the superprocess $Y$ or indeed the continuous-state branching process $||Y_D||$. Specifically, for the case of $Y$, the motion of the backbone, $Z^Y$, is that of a Brownian motion with drift $-\rho$ and $\psi$ remains the same. For the case of $||Y_D||$, we may consider the motion process to be that of a particle which remains fixed at a point and the branching mechanism $\psi$ is replaced by $\psi_D$.


\section{Proof of Theorem \ref{I}}\label{pfI}
Proof of (i):  Using obvious notation in light of (\ref{notation}),
and referring to the discussion following Theorem \ref{III}, we
necessarily have that $R^{X}_t$ is equal in law to $ \inf\{y>0 :
Y_{D^t_{0}}|_{(0,\infty)\times \{t\}}(y,\infty) = 0\}, $ and the
latter is $\mathbb{Q}_{\delta_x}$-almost surely bounded above by
$R^Y_t$. It is known from Corollary 3.2 of \cite{KLMSR} that, under
(\ref{A3}),   for any $\rho\leq \sqrt{2\alpha}$,
\[
\lim_{t\rightarrow\infty}\frac{R^Y_t}{t}=\sqrt{2\alpha}-\rho
\]
$\mathbb{Q}_{\delta_x}$-almost surely on the survival set of $Y$.
It follows that, under the same assumptions,
\begin{equation}
\limsup_{t\rightarrow\infty}\frac{R^{X}_t}{t}
\leq \sqrt{2\alpha}-\rho\,\mbox{ on}\,\{\zeta^{X}_+=\infty\},
\label{upper}
\end{equation}
$\mathbb{P}_{\delta_x}$-almost surely.

For the lower bound, note that the backbone decomposition allows us to deduce straight away that, again using obvious notation, on $\{\zeta^{\Lambda^X}_+=\infty\}$, $R^{Z^X}_t\leq R^{\Lambda^X}_t$ holds ${\mathbf P}_{\delta_x}$-almost surely for each $x>0$.
The restriction of the process $Z^X$ to $(0,\infty)$ can be formally identified as a branching Brownian motion with killing at the origin.
In  \cite{HHK} it was shown that a dyadic branching Brownian motion with drift $-\rho$ and killing at the origin which branches at rate $q$ has the property that the right most particle speed is equal to $\sqrt{2q} -\rho$ on survival. In fact careful inspection of their proof shows that it is straightforward to replace dyadic branching by a random number of offspring with mean $m\in(1,\infty)$. In that case the right most speed is equal to $\sqrt{2q(m-1)}-\rho$. Note that for the process $Z^X$, we easily compute from (\ref{F}) that $q(m-1) = \alpha$. 
We now have, 
that
\begin{equation}
 \liminf_{t\rightarrow\infty} \frac{R^{\Lambda^X}_t}{t}\geq \lim_{t\rightarrow\infty}\frac{R^{Z^X}_t}{t}=\sqrt{2\alpha}-\rho\mbox{ on } \{\zeta^{Z^X}_+=\infty\},
 \label{liminf}
\end{equation}
$\mathbf{P}_{\delta_x}$-almost surely. Let us temporarily assume however that  $\{\zeta^{Z^X}_+<\infty\}$  agrees with the event $\{\zeta^{\Lambda^X}_+<\infty\}$ under $\mathbf{P}_{\delta_x}$. Theorem \ref{main-1} now allows us to conclude from (\ref{liminf}) that
\[
\liminf_{t\rightarrow\infty} \frac{R^{X}}{t} \geq \sqrt{2\alpha} - \rho\mbox{ on }\{\zeta^{X}_+ = \infty\}
\]
$\mathbb{P}_{\delta_x}$-almost surely.

To  complete the proof of part (i) we must therefore show that $\{\zeta^{Z^X}_+<\infty\}$  agrees with the event $\{\zeta^{\Lambda^X}_+<\infty\}$ under $\mathbf{P}_{\delta_x}$ and that their common probabilities give the unique solution to (\ref{FKPP}).
To this end, first note that the backbone decomposition, and in particular the Poisson embedding of $Z^X$ in $\Lambda^X$,  gives us that $\{\zeta^{\Lambda^X}_+<\infty\} \subseteq \{\zeta^{Z^X}_+<\infty\}$.
Next note that the backbone decomposition also tells us that $\{Z^X_0(0,\infty) =0 \}\subseteq\{\zeta^{\Lambda^X}_+<\infty\}$. If we define the monotone increasing function $\Phi:[0,\infty)\rightarrow [0,\infty)$ by
\[
e^{-\Phi( x)} = \mathbb{P}_{\delta_x}(\zeta^{X}_+<\infty) =
 \mathbf{P}_{\delta_x}(\zeta^{\Lambda^X}_+<\infty),
\]
so that in particular $\Phi(0) = 0$,
then the previous observations tell us that for $x>0$,
\[
e^{-\lambda^*}\leq e^{-\Phi(x)}\leq  \mathbf{P}_{\delta_x}(\zeta^{Z^X}_+<\infty)<1.
\]
{\color{black} Note that the final inequality above is strict as all initial particles in $Z^X$ may hit the {\color{black}stopping} boundary before branching with positive probability.}
 It is a straightforward exercise to show, using the Markov branching property and the fact that $\Phi(0) = 0$, that $\Phi$ respects the relation
 \begin{equation}
e^{-\Phi(x)} = \mathbb{E}_{\delta_x}(\mathbb{P}_{X_t}(\zeta^X_+ <\infty)) = \mathbb{E}_{\delta_x}(e^{-\l \Phi, X^+_t\r})
{\color{black}= \mathbb{E}_{\delta_x}(e^{-\l \Phi, X_t\r})}
\text{ for all }x, t\geq 0.
\label{mmgfn}
\end{equation}
Inspecting the semi-group evolution equation (\ref{PDE}) for $X$ with data $f  = \Phi$  and taking account of the fact that  its unique solution given by (\ref{prePDE}), we see that $\Phi$ solves the  {\color{black} differential  equation} in  (\ref{FKPP}).

To show that $\Phi(+\infty) = \lambda^*$, note that the law of $\l \Phi, X_t^+\r$ under $\mathbb{P}_{\delta_x}$ is equal to that of $\l \Phi(x+\cdot), Y_{D^t_{-x}}|_{(-x,\infty)\times\{t\}}\r$ under $\mathbb{Q}_{\delta_0}$. Thanks to the monotonicity of $\Phi(x)$ and $Y_{D^t_{-x}}|_{(-x,\infty)\times\{t\}}$ in $x$ and the fact that $0<\Phi(x)\leq \lambda^*$ we have, with the help of dominated convergence,
\begin{equation}
e^{-\Phi(+\infty)} = \lim_{x\uparrow\infty} \mathbb{Q}_{\delta_0}(e^{- \l\Phi(x+\cdot), Y_{D^t_{-x}}|_{(-x,\infty)\times \{t\}} \r}) = \mathbb{Q}_{\delta_0}(e^{-\Phi(\infty) ||Y_t||}).
\label{mustbe}
\end{equation}
On account of the fact that the process $\{||Y_t||: t\geq 0\}$ is a continuous-state branching process with branching mechanism $\psi$ the equality in (\ref{mustbe}) together with the fact that $\Phi(+\infty)\in(0,\lambda^*]$ forces us to deduce that $\Phi(+\infty) = \lambda^*$.

 Now suppose that $\phi$ solves (\ref{FKPP}). The backbone decomposition  tells us that for all $t\geq0$, $Z_t(\cdot)$ given $\Lambda^X_t(\cdot)$ is a Poisson random field with intensity measure
$\lambda^* \Lambda^X_t(\cdot)$. Hence,
\begin{eqnarray*}
\e_{\delta_x}\left[e^{\langle\log(1-\phi/\lambda^*),Z^X_t\rangle}\right]
&=&\e_{\delta_x}\e\left[e^{\langle\log(1-\phi/\lambda^*),Z^X_t\rangle}\bigg|\Lambda^X_t\right]\\
&=&
\e_{\delta_x}\left[\exp\left\{-\int\left(1-e^{\log(1-\phi(y)/\lambda^*)}\right)
\lambda^*\Lambda^X_t({\rm d}y)\right\}\right]
\\&=&
\e_{\delta_x}\left[e^{-\langle\phi,\Lambda^X_t\rangle}\right]\\
&=&e^{-\phi(x)}.
\end{eqnarray*}
Recalling  that $\phi$ is monotone with  $\phi(0+) = 0$ and $\phi(+\infty) = \lambda^*$, and hence that $-\log(1- \phi/\lambda^*)\in[0,\infty)$ so that  $$\langle-\log(1-\phi/\lambda^*),Z^X_t\rangle\geq -\log (1-\phi(R^{Z^X}_t)),$$
 it follows with the help of the known asymptotics of $R_t^{Z^X}$, eg (\ref{liminf}), that
\begin{eqnarray*}
\limsup_{t\rightarrow\infty}\e_{\delta_x}\left[e^{\langle\log(1-\phi/\lambda^*),Z^X_t\rangle}\mathbf{1}_{\{\zeta^{Z^X}_+=\infty\}}\right]&\leq &
\limsup_{t\rightarrow\infty}\e_{\delta_x}\left[ e^{\log (1 - \phi(R^{Z^X}_t)/\lambda^*)}\mathbf{1}_{\{\zeta^{Z^X}_+=\infty\}}\right]=0.
\end{eqnarray*}
Subsequently
\begin{eqnarray}
e^{-\phi(x)}
&=&\mathbf{P}_{\delta_x}(\zeta^{Z^X}_+<\infty)
 +\lim_{t\rightarrow \infty}\e_{\delta_x}\left[e^{-\langle-\log(1-\phi/\lambda^*),Z^X_t\rangle}\mathbf{1}_{\{\zeta^{Z^X}_+=\infty\}}\right]\notag\\
&=&\mathbf{P}_{\delta_x}(\zeta^{Z^X}_+<\infty).
\label{toconclude}
\end{eqnarray}
We conclude from (\ref{toconclude}) that
\[
\Phi(x) = -\log \mathbf{P}_{\delta_x}(\zeta^{\Lambda^X}_+<\infty) =-\log  \mathbf{P}_{\delta_x}(\zeta^{Z^X}_+<\infty)
\]
is the unique monotone solution to (\ref{FKPP}). Moreover,  since $\{\zeta^{\Lambda^X}_+<\infty\} \subseteq \{\zeta^{Z^X}_+<\infty\}$,  we may now also deduce that $\{\zeta^{\Lambda^X}_+<\infty\} = \{\zeta^{Z^X}_+<\infty\}$, $\mathbf{P}_{\delta_x}$-almost surely, which completes the proof of part (i) of the Theorem.

\bigskip

\noindent Proof of (ii): Suppose now that $\rho\geq \sqrt{2\alpha}$.
The estimate $R^{X}_t\leq R^Y_t$ used in (\ref{upper}) now tells us
that $\mathbb{P}_{\delta_x}(\zeta^{X}_+<\infty) = 1$ and hence,
because of the backbone decomposition, it also tells us that
$\mathbf{P}_{\delta_x}(\zeta^{\Lambda^X}_+<\infty) =1$. As noted
earlier, the Poisson embedding of $Z^X$ in $\Lambda^X$ gives us that
$\{\zeta^{\Lambda^X}_+<\infty\}\subseteq\{\zeta^{Z^X}_+<\infty\}$
and hence it follows that
$\mathbf{P}_{\delta_x}(\zeta^{Z^X}_+<\infty) =1$. Suppose now that a monotone
wave solution, $\Phi$, to (\ref{FKPP}) exists. Then the computation
in (\ref{toconclude}) forces us  to conclude that $\Phi\equiv 0$
which is a contradiction. Therefore there can be no solutions to
(\ref{FKPP}). \hfill$\square$

\begin{remark}\label{trick}\rm Whilst Theorem \ref{I} offers results on the existence and uniqueness of solutions to (\ref{FKPP}), Proposition 2 of Pinsky \cite{Pin} and {\color{black} Theorem 1} of Harris et al. \cite{HHK} also offer the rate of decay of monotone solutions at $+\infty$ to the wave equation
\begin{equation}
\frac{1}{2}\Psi''(x) - \rho\Psi'(x) + F(\Psi(x)) \text{ on }x>0\text{ with }\Psi(0+) = 1\text{ and }\Psi(+\infty) = 0,
\label{HarrisFKPP}
\end{equation}
for $\rho<\sqrt{2q}$ where $F(s) = q(s^2-s)$ and $q>0$.
A straightforward inspection of the proof in Theorem 1 of  Harris et al. \cite{HHK} shows that in fact their result on the decay of $\Psi$ holds for more general functions $F$ taking, for example, the form $F(s) = q(\sum_{n=2}^\infty s^n p_n - s)$ for $s\in[0,1]$, $q>0$ where $\{p_n: n\geq 2\}$ is a probability distribution satisfying $F'(1)<\infty$. {\color{black} Specifically, most of the arguments in \cite{HHK} do not require a dyadic offspring distribution such as is imposed there, however, in Section 6 one must take care with the exponential term in the martingale defined in (14). In their terminology, the integrand in the exponential term, $\beta (f(Y_s) - 1)$, needs to be replaced by $G(Y_s)$ where $G(s) = F(s)/s$. Thereafter, the necessary adjustments, which pertain largely to bounds, are relatively obvious.}
 In that case their result reads as follows. For all $\rho<\sqrt{2F'(1)}$
\[
\lim_{x\uparrow\infty}e^{-(\rho- \sqrt{\rho^2+ 2q})x }\Psi(x) = k_\rho
\]
for some $k_\rho\in(0,\infty)$.

Note that when $F$ is given by (\ref{F}) it is straightforward to check that $\Psi$ solves (\ref{HarrisFKPP}) if and only if $\lambda^*(1-\Psi)$ solves (\ref{FKPP}). It follows immediately that when $\rho<\sqrt{2\alpha}$
\[
\lim_{x\uparrow\infty}\exp\left\{-\left(\rho- \sqrt{\rho^2+ 2\psi'(\lambda^*)}\right)x \right\}(1-\Phi(x)/\lambda^*) = k_\rho.
\]\end{remark}

\section{Proof of Theorem \ref{III}}\label{pfIII}

As alluded to above, our objective is to embed an existing result for branching Brownian motion with absorption at the origin into the superprocess setting with the help of the backbone decomposition. For all $x\in\mathbb{R}$ we shall denote by ${\rm Q}_{\delta_x}$ the law of the backbone decomposition, $(Z^Y, \Lambda^Y)$ of $Y$. The existing result in question is due to Maillard \cite{Mal} and when paraphrased in terms of the backbone process $Z^Y$ for $Y$, states that, under the condition that $\sum_{n\geq 2}n(\log n)^{2+\varepsilon}p_n<\infty$, for some $\varepsilon>0$,  and $\rho = \sqrt{2\alpha}$ it follows that for all $x\geq -z$,
\begin{equation}
{\rm Q}_{\delta_x}(||Z^Y_{D_{-z}}||>n|||Z^Y_0|| =1)\sim\frac{\sqrt{2\alpha}(x+z)e^{\sqrt{2\alpha} (x+z)}}{n(\log n)^2}
\label{Mal-asym}
\end{equation}
as $\mathbb{N} \ni n\uparrow\infty$. Here we understand $Z^Y_{D_{-z}}$ to mean the atomic valued measure, supported on $\{-z\}\times[0,\infty)$ which describes the space-time position of particles in the branching diffusion $Z^Y$ who are first in their line of descent to exit the domain $D^\infty_{-z}$. The process $||Z^Y_D||: = \{||Z^Y_{D_{-z}}||: z\geq 0\}$ is known to be a continuous time Galton-Watson (cf. Lemma 3.1 and Proposition 3.2 in \cite{Mal} or Proposition 3 in  \cite{neveu}), which, like its continuous-state space analogue $||Y_D||$, has growth rate $\sqrt{2\alpha}$. Maillard's result follows by first establishing that
\begin{equation}
F_z(s) : = {\rm Q}_{\delta_0}(s^{||Z^Y_{D_{-z}}||} | ||Z^Y_0|| =1)
\label{Fz-def}
\end{equation}
satisfies
\begin{equation}
F_z''(1-s)\sim\frac{\sqrt{2\alpha}ze^{z\sqrt{2\alpha}}}{s(\log(1/s))^2},\qquad\text{as $s\downarrow0$},\label{Fz}
\end{equation}
and then applying a classical Tauberian theorem.

The strategy for our proof of Theorem \ref{III} will be to first show that the moment condition $\sum_{n\geq 2}n(\log n)^{2+\epsilon}p_n<\infty$ is implied by (\ref{Pi-moment}). Thereafter, we shall appeal to an analytical identity that arises through the Poissonisation property of the backbone decomposition, thereby allowing us to convert the asymptotic (\ref{Fz}) into an appropriate asymptotic which leads, again through an application of a Tauberian theorem, to the conclusion of  Theorem \ref{III}. We start with the following lemma. 

\begin{lemma}
If $\int_{[1,\infty)} x(\log x)^{2+\varepsilon}\Pi({\rm d}x)<\infty$ for some $\varepsilon>0$, then $\sum_{n\geq 2}n(\log n)^{2+\varepsilon}p_n<\infty$.
\label{integraltest}
\end{lemma}

\begin{proof}
Appealing to the definition of $\{p_n: n\geq 2\}$  it suffices to prove that
\begin{equation}\label{sum}
\int_0^{\infty}\sum_{n\geq 2}n(\log n)^{2+\varepsilon}\frac{(\lambda^*x)^n}{n!}e^{-\lambda^*x}\Pi({\rm d}x)<\infty.
\end{equation}
To this end, define the following function $f(x)=(\log (1+x))^{2+\varepsilon}$, then it is easy to see that
\begin{equation}
f''(x)=-(2+\varepsilon)\frac{(\log(1+x))^{\varepsilon}}{(1+x)^2}(\log (1+x)-(1+\varepsilon)).\notag
\end{equation}
Then we can find $N_0\in\mathbb{N}$ such that $\log(1+N_0)>(1+\varepsilon)$ and subsequently that $f''(x)<0$ for $x\geq N_0$. This implies  that $f$ is concave in $(N_0,\infty)$. Hence, using Jensen's inequality
\begin{eqnarray}
\lefteqn{\sum_{n\geq N_0+1}n(\log n)^{2+\varepsilon}\frac{(\lambda^*x)^n}{n!}e^{-\lambda^*x}}&&\\
&=&\lambda^*x\sum_{n\geq N_0}(\log (n+1))^{2+\varepsilon}\frac{(\lambda^*x)^n}{n!}e^{-\lambda^*x}\notag\\
&\leq&(\lambda^* x)\left(\sum_{n\geq N_0}e^{-\lambda^* x}\frac{(\lambda^*x)^n}{n!}\right)\left(\log\left(\frac{\sum_{n\geq N_0}n\frac{(\lambda^*x)^n}{n!}e^{-\lambda^*x}}{\sum_{n\geq N_0}\frac{(\lambda^*x)^n}{n!}e^{-\lambda^*x}}+1\right)\right)^{2+\varepsilon}\notag\\
&\leq& (\lambda^* x)\left(\log\left(\frac{\sum_{n\geq
N_0}n\frac{(\lambda^*x)^n}{n!}}{\sum_{n\geq
N_0}\frac{(\lambda^*x)^n}{n!}}+1\right)\right)^{2+\varepsilon}.\label{c2}
\end{eqnarray}
On the other hand we have that
\begin{equation}
\lim_{x\to\infty}\frac{\log\left(\frac{\sum_{n\geq
N_0}n\frac{(\lambda^*x)^n}{n!}}{\sum_{n\geq
N_0}\frac{(\lambda^*x)^n}{n!}}+1\right)}{\log x}=1.\notag
\end{equation}
So we can find $K>0$ such that if $x>K$
\begin{equation}
\log\left(\frac{\sum_{n\geq
N_0}n\frac{(\lambda^*x)^n}{n!}}{\sum_{n\geq
N_0}\frac{(\lambda^*x)^n}{n!}}+1\right)<2\log x.\notag
\end{equation}
 Using (\ref{c2}), this implies that
\begin{equation}
\int_{[K,\infty)}\sum_{n\geq N_0+1}n(\log n)^{2+\varepsilon}\frac{(\lambda^*x)^n}{n!}e^{-\lambda^*x}\Pi({\rm d}x)<2^{2+\varepsilon}\lambda^*\int_{[K,\infty)}x(\log x)^{2+\varepsilon} \Pi({\rm d}x)<\infty.\notag
\end{equation}
On the other hand, by choosing $N_0$ large enough, we also have that for all $n\geq N_0+1$, $(\log n)^{2+\varepsilon}< C(n-1)$ for some $C>0$. Hence
\begin{eqnarray*}
\lefteqn{\int_{(0,K)}\sum_{n\geq N_0+1}n(\log n)^{2+\varepsilon}\frac{(\lambda^*x)^n}{n!}e^{-\lambda^*x}\Pi({\rm d}x)}&&\\
&\leq& (\lambda^*K)^2\sum_{n\geq N_0+1 }\int_{(0,\infty)}\frac{(\lambda^*x)^{n-2}}{(n-2)!}e^{-\lambda^*x} \Pi({\rm d}x)\\
&\leq& (\lambda^*K)^2\sum_{n\geq 2 }p_n<\infty
\end{eqnarray*}

For the first terms of (\ref{sum}) we have
\begin{equation*}
\int_{(0,\infty)}\sum_{n=2}^{N_0}n(\log
n)^{2+\varepsilon}\frac{(\lambda^*x)^n}{n!}e^{-\lambda^*x}\Pi(dx)=\sum_{n=2}^{N_0}n(\log
n)^{2+\varepsilon}\frac{(\lambda^*)^n}{n!}\int_{(0,\infty)}x^ne^{-\lambda^*x}\Pi(dx)<\infty,
\end{equation*}
which follows from the fact that each term of the sum is finite.
This completes the proof.
\ter\end{proof}

\begin{remark}\rm
It is not difficult to show that the converse of the statement in Lemma \ref{integraltest} is also true, however we leave it as an exercise for the reader.
\end{remark}

Let us now turn to the proof of Theorem \ref{III}. We approach the proof here on in two steps. The first step is to show that the process $||Z^Y_D||$ under ${\rm Q}_{\delta_x}$, for which distributional properties are known thanks to (\ref{Mal-asym}), has the same branching generator  as the continuous-time Galton-Watson process $Z':=\{Z'_z: z\geq 0\}$, where the latter is the backbone embedded in the continuous-state branching process $||Y_D||$. Thanks to the backbone decomposition of $||Y_D||$, say $(Z',\Lambda')$,  and the easily seen fact that $\psi_D(\lambda^*) = 0$, we have that the law of $Z'_{z}$ given $\Lambda'_{z}$ is that of a Poisson random variable with parameter $\lambda^*\Lambda'_z$. This Poissonisation result will allow us to feed the known distributional asymptotic for $Z'_z$ (equiv. $||Z^Y_{D_{-z}}||$) into the required result for $\Lambda'_z$ (equiv. $||Y_{D_{-z}}||$).

\bigskip

{\it\underline{Step 1:}} We start by recalling from  Maillard
\cite{Mal},  Section 3, that $||Z^Y_{D}||$ has branching generator
given by $\Theta'(\Theta^{-1}(s)), \, s\in[0,1]$ where $\Theta$ is
the unique monotone solution to  the wave equation
\begin{equation}
\frac{1}{2}\Theta''(x) - \sqrt{2\alpha}\Theta'(x) +F(\Theta(x)) = 0 \text{ on }\mathbb{R}\text{ with }\Theta(-\infty) = 1\text{ and }\Theta(+\infty)=0,
\label{THETA}
\end{equation}
and $F$ is the branching generator of the backbone $Z^Y$ given in (\ref{F}). Again appealing to (\ref{F}) but for the backbone decomposition $(Z',\Lambda')$ of $||Y_{D}||$ and the fact that $\psi_D(\lambda^*)=0$, we know that $Z'$ has branching generator given by
\[
F_D(s) = \frac{1}{\lambda^*}\psi_D(\lambda^*(1-s)), \text{ for }s\in[0,1].
\]
Our objective is thus to show that $F_D(s) =\Theta'(\Theta^{-1}(s))$ for all $s\in[0,1]$.

To this end, recall that $\psi_D(\lambda) = \Psi'(\Psi^{-1}(\lambda))$  for $\lambda\in[0,\lambda^*]$ where $\Psi$ solves (\ref{PSI}). It is a straightforward exercise to check that  $\Theta(x) = 1-\Psi(-x)/\lambda^*$. Indeed  ${\Theta}(+\infty) = 0$ and ${\Theta}(-\infty)=1$ and  $\Theta$ solves (\ref{THETA}) on account of the fact that $\Psi$ solves (\ref{PSI}). Moreover, one readily confirms that
\[
\frac{1}{\lambda^*}\Psi'(\Psi^{-1}(\lambda^*(1-s))) = \Theta'(\Theta^{-1}(s))\text{ for }s\in[0,1],
\]
This implies in turn that the required equality, $F_D(s) =\Theta'(\Theta^{-1}(s))$, holds and in particular that $||Z^Y_{D}||$ and $Z'$ have the same branching generator.

\bigskip

{\it\underline{Step 2:}} Recall that $(Z',\Lambda')$ is the  backbone decomposition of $\{||Y_{D_{-z}}||:z\geq 0\}$ and  denote the law of former by $\mathcal{Q}_{x}$  when the latter has law $\mathbb{Q}_{\delta_x}$.  Appealing to spatial homogeneity, we may henceforth proceed without loss of generality by assuming that $x =0$.

It follows from the conclusion of Step 1 and the Poissonisation property of the backbone decomposition that for $z\geq 0$ and $s\in[0,1]$,
\begin{equation}
{\rm Q}_{\delta_0}(s^{||Z^Y_{D_{-z}}||}) = \mathcal{Q}_{0}(s^{Z'_z}) =\mathcal{Q}_{0}\left(\mathcal{Q}_{0}(s^{Z'_z}|\Lambda'_z)\right) = \mathcal{Q}_{0}(e^{-\lambda^*\Lambda'_z(1-s)}) = \mathbb{Q}_{\delta_0}(e^{-\lambda^*||Y_{D_{-z}}||(1-s)}).
\label{lotsofQs}
\end{equation}
Now using the fact that, under ${\rm Q}_{\delta_0}$, $||Z^Y_{D_{0}}|| = ||Z^Y_0||$ is a Poisson random variable with intensity $\lambda^*$ we have
\begin{equation}\label{tlp2}
{\rm Q}_{\delta_0}\left[s^{||Z^Y_{D_{-z}}||}\right]=\sum_{k=0}^{\infty}e^{-\lambda^*}\frac{(\lambda^*)^k}{k!}F_z(s)^k=\exp\left\{-\lambda^*\left(1-F_z(s)\right)\right\},
\end{equation}
where $F_z(s)$ was defined in (\ref{Fz-def}). 
If we set
\[
w_z(s)  = \mathbb{Q}_{\delta_0}(e^{-s||Y_{D_{-z}}||})
\]
then (\ref{lotsofQs}) and (\ref{tlp2}) tell us that
\begin{equation*}\label{re}
w_z(\lambda^*s)=\exp\{-\lambda^*(1-F_z(1-s))\}.
\end{equation*}

Taking second derivatives on both sides of the last equality gives us
\begin{equation}\label{sd}
(\lambda^*)^2w''_z(\lambda^*s)=(\lambda^*F''_z(1-s)+(\lambda^*F'_z(1-s))^2)\exp\{-\lambda^*(1-F_z(1-s))\}.
\end{equation}
Recalling (\ref{Fz}) and noting that
\begin{equation}
F_z'(1-s)\sim e^{z\sqrt{2\alpha}},\qquad\text{as $s\downarrow0$},\notag
\end{equation}
which holds on account of the fact that $||Z^Y_D||$ is a continuous-time Galton-Watson process with growth rate $\sqrt{2\alpha}$,
we have from (\ref{sd})  that
\begin{equation}\label{as1}
w_z''(\lambda^*s)\sim\frac{\sqrt{2\alpha}ze^{z\sqrt{2\alpha}}}{\lambda^*s(\log(1/s))^2},\qquad\text{as $s\downarrow0$}.
\end{equation}
Taking $u=\lambda^*s$ and using (\ref{as1}) we obtain
\begin{equation}
w_z''(u)\sim\frac{\sqrt{2\alpha}ze^{z\sqrt{2\alpha}}}{u(\log\lambda^*+\log(1/u))^2}\sim\frac{\sqrt{2\alpha}ze^{z\sqrt{2\alpha}}}{u(\log(1/u))^2} ,\qquad\text{as $u\downarrow0$}.\label{as2}
\end{equation}
Denote by $U_z({\rm d}y)$ the measure in $(0,\infty)$ defined by the relation
\begin{equation}
w_z(s)=\int_{[0,\infty)}e^{-sy}U_z({\rm d}y).\notag
\end{equation}
In other words $U_z({\rm d}y)=\mathbb{Q}_{\delta_0}(||Y_{D_{-z}}||\in {\rm dy})$ for $y\geq 0$. And let us take $\tilde{U}_z({\rm d}y)=y^2U_z({\rm d}y)$ on $[0,\infty)$, then it is easy to see that
\begin{equation}
w_z''(s)=\int_{[0,\infty)}e^{-sy}\tilde{U}_z({\rm d}y).\notag
\end{equation}
Then using Theorem 2 XIII.5 in \cite{fe} we have using (\ref{as2}), that
\begin{equation}
\tilde{U}_z(t)\sim\frac{t\sqrt{2\alpha}ze^{z\sqrt{2\alpha}}}{(\log t)^2}\qquad\text{as $t\to\infty$}.\notag
\end{equation}
Using integration by parts it is easy to see that
\begin{equation}
U_z(t,\infty)=\int_{(t,\infty)}y^{-2}\tilde{U}_z({\rm d}y)\sim \sqrt{2\alpha}ze^{z\sqrt{2\alpha}}\left(2\int_t^{\infty}\frac{1}{y^2(\log y)^2}{\rm d}y-\frac{1}{t(\log t)^2}\right)\qquad\text{as $t\to\infty$}.\notag
\end{equation}
But by Theorem 1 VIII.9 in \cite{fe} the  integral on the right hand side above is equivalent to $1/t(\log t)^2$.  This implies that
\begin{equation}
\mathbb{Q}_{\delta_0}(||Y_{D_{-z}}||>t)=U_z(t,\infty)\sim\frac{\sqrt{2\alpha}ze^{z\sqrt{2\alpha}}}{t(\log t)^2}\qquad\text{as $t\to\infty$},\notag
\end{equation}
which proves the result.
\ter

\begin{remark}\rm
Maillard \cite{Mal} gives further results in the case that
$\rho>\sqrt{2\alpha}$ for the asymptotic behaviour of
$\mathbb{Q}_{\delta_x}(||Y_{D_{-z}}||>t)$ as $t\uparrow\infty$.
Again using ideas of Poisson embedding through the backbone, analogous asymptotics can be transferred from the case
of branching Brownian motion to super-Brownian motion.
\end{remark}

\section*{Acknowledgements}
All three authors would like to thank and anonymous referee for
their comments and remarks on an earlier draft of this paper which
lead to its improvement. AEK acknowledges financial support from the
Santander Research Grant Fund, JLP acknowledges financial support
from CONACyT-MEXICO grant number 150645, AMS acknowledges financial
support from CONACyT-MEXICO grant number 129076.


\begin{thebibliography}{55}\itemsep=-4pt\small

\bibitem{AB} Addario-Berry, L. and Broutin, N. (2009): Total progeny in killed branching random walk. \texttt{http://arxiv.org/abs/0908.1083v1}

\bibitem{AHZ} A\"id\'ekon, E., Hu, Y. and Zindy, O. (2011): The precise tail behavior of the total progeny of a killed branching random walk.
\texttt{arxiv.org:1102.5536}

\bibitem{BKMS} Berestycki, J., Kyprianou, A.E. and Murillo-Salas, A. (2011): The prolific backbone decomposition for supercritical superdiffusions. {\it Stoch. Proc. Appl.} {\bf 121}, 1315--1331.
\bibitem{Dyn1991} Dynkin, E.B.  (1991): A probabilistic approach to one class of non-linear differential equations. {\it Probab. Th. Rel. Fields} {\bf 89}, 89--115.


\bibitem{D}  Dynkin, E.B.  (1991):         Branching particle systems and superprocesses, {\em  Ann. Probab.} {\bf 19}, 1157--1194.


\bibitem{Dyn1993} Dynkin, E.B. (1993): Superprocesses and Partial Differential Equations. {\it Ann. Probab.} {\bf 21}, 1185--1262.
\bibitem{Dyn2001} Dynkin, E.B. (2001): Branching exit Markov systems and superprocess. {\it Ann. Probab.} {\bf 29}, 1833--1858.
\bibitem{Dyn2002} Dynkin, E.B. (2002): {\it Diffusions, Superdiffusions and Partial Differential Equations.} AMS, Providencem R.I.
\bibitem{Dynkin-Kuznetsov} Dynkin, E.B. and Kuznetsov, S.E. (2004): $\mathbb{N}$-measures for branching Markov exit systems and their applications to differential equations. {\it Prob. Theory Related Fields} {\bf 130}, 135--150.

\bibitem{El} El Karoui, N. and  Roelly, S. (1991)
Propri\'et\'es de martingales, explosion et repr\'esentation de L\'evy-Khintchine d'une classe de processus de branchement \`a valeurs mesures. {\it Stoch. Proc.. Appl.} {\bf 38} (1991), 239--266.

\bibitem{EP}  Engl\"{a}nder, J. and Pinsky, R.G. (1999): On the construction
and support properties of measure-valued diffusions on $D\subseteq
R^{d}$ with spatially dependent branching. \textit{Ann. Probab.}
{\bf 27} 684--730.

\bibitem{EO}  Evans, S. N. and O'Connell, N. (1994): Weighted Occupation
Time for Branching Particle Systems and a Representation for the
Supercritical Superprocess. \textit{Canad. Math. Bull.}, {\bf 37}
187-196.
\bibitem{fe} Feller, W. (1971): {\it  An introduction to probability theory and its applications. Vol. II. 2nd. ed.} Wiley Series in probability and Mathematical Statistics. New York etc.: John Wiley and Sons, Inc. XXIV, 669 p.

              \bibitem{Fitz}   Fitzsimmons, P.J.  (1988): Construction and regularity of measure-valued Markov branching processes. {\it Israeli J. Math.} {\bf 64}, 337--361.


\bibitem{HHK} Harris, J.W., Harris, S.C. and Kyprianou, A.E.  (2006): Further probabilistic analysis of the Fisher-Kolmogorov-Pretrovskii-Piscounov: one sided travelling-waves. {\it  Annales de l'Instut Henri Poincar\'e}, {\bf 42}, 125--145.


\bibitem{Kam}  Kametaka, Y. (1976): On the nonlinear diffusion equation of Kolmogorov-Petrovskii-Piskunov type. {\it Osaka J. Math.} {\bf 13},  11--66.

\bibitem{K}  Kyprianou, A.E. (2006): {\it Introductory lectures on fluctuations of L\'evy processes with applications.} Springer.

\bibitem{KLMSR} Kyprianou, A., Liu, R.-L., Murillo-Salas, A. and Ren, Y.-X. (2011): Supercritical super-Brownian motion with a general branching mechanism and travelling waves. {\it To appear in  Annales de l'Instut Henri Poincar\'e}.

\bibitem{LG} Le Gall, J-F. (1999): {\it Spatial branching processes, random snakes and partial differential equations.} Lectures in Mathematics ETH Z\"urich. Birkh\"auser Verlag, Basel.


\bibitem{Mal} Maillard, P.  (2011): The number of absorbed individuals in branching Brownian motion with a barrier. \texttt{arXiv:1004.1426}



\bibitem{neveu}  Neveu, J.  (1988): Multiplicative martingales for spatial
branching processes. In Seminar on Stochastic Processes 1987, eds E. \c{C}%
inlar, K.L. Chung, R.K. Getoor. Progress in Probability and Statistics, 15,
223--241. Birkha\"{u}ser, Boston.

\bibitem{Pin} Pinsky, R.G.  (1995): K-P-P-type asymptotics for nonlinear diffusion in a large ball with infinite boundary data and on $\bold R^d$ with infinite initial data outside a large ball. {\it Comm. Partial Differential Equations} {\bf 20},  1369--1393.

\bibitem{SV1} Salisbury, T. and Verzani, J. (1999): On the conditioned exit measures of super Brownian motion.
{\it Prob. Theory Relat. Fields} {\bf 115}, 237-285

\bibitem{SV2} Salisbury, T. and Verzani, J.  (2000): Non-degenerate conditionings of the exit measure of super Brownian motion.
{\it Stoch. Proc.  Appl.} {\bf 87}, 25-52.

\bibitem{Sheu} Sheu, Y.C. (1997): Lifetime and compactness of the range of a $\psi$-super-Brownian motion with general branching mechanism. {Stoch. Proc. Appl.} {\bf 70}, 129--141.

\bibitem{U78}  Uchiyama, K.  (1978): The behavior of solutions of some fnon-linear diffusion equations for large time. {\em J. Math. Kyoto Univ.} {\bf 18}, 453-508.

\bibitem{watanabe1968} Watanabe, K.  (1968): A limit theorem of branching processes and continuous-state branching processes. {\em J. Math. Kyoto Univ.} {\bf 8}, 141--167.
\end{thebibliography}
\end{document}